\documentclass[a4paper]{amsart}

\usepackage{amsmath}
\usepackage{amssymb}
\usepackage{amsfonts}
\usepackage{mathrsfs}
\usepackage{array}
\usepackage{enumerate}
\usepackage[french, english]{babel}
\usepackage[all]{xy}

\numberwithin{equation}{section}

\newcommand{\msg}{\mathscr{G}}
\newcommand{\msi}{\mathscr{I}}\newcommand{\msl}{\mathscr{L}}

\newcommand{\mst}{\mathscr{T}}

    \newcommand{\BC}{{\mathbb {C}}} 
     
    \newcommand{\BG}{{\mathbb {G}}}

    \newcommand{\BQ}{{\mathbb {Q}}}

     \newcommand{\BZ}{{\mathbb {Z}}}

     \newcommand{\CL}{{\mathcal {L}}}
     
    \newcommand{\CO}{{\mathcal {O}}}

    \newcommand{\CW}{{\mathcal {W}}}

     \newcommand{\fb}{{\mathfrak{b}}}

    \newcommand{\fm}{{\mathfrak{m}}} 
     \newcommand{\fp}{{\mathfrak{p}}}
    \newcommand{\fq}{{\mathfrak{q}}} 
    \newcommand{\fs}{{\mathfrak{s}}} 
     
     \newcommand{\fx}{{\mathfrak{x}}}

    \newcommand{\fC}{{\mathfrak{C}}}

    \newcommand{\fM}{{\mathfrak{M}}} 
     \newcommand{\fP}{{\mathfrak{P}}}
     \newcommand{\fR}{{\mathfrak{R}}}

    \newcommand{\fW}{{\mathfrak{W}}}

    \newcommand{\Coker}{{\mathrm{Coker}}}

    \newcommand{\End}{{\mathrm{End}}}

    \newcommand{\Gal}{{\mathrm{Gal}}} 
    
    \newcommand{\Hom}{{\mathrm{Hom}}}
    
    \renewcommand{\Im}{{\mathrm{Im}}}

    \newcommand{\ord}{{\mathrm{ord}}}

    \renewcommand{\mod}{\ \mathrm{mod}\ }

    \newcommand{\Res}{{\mathrm{Res}}}
    \newcommand{\Sel}{{\mathrm{Sel}}}

    \font\cyr=wncyr10

    \newcommand{\Sha}{\hbox{\cyr X}}
    \newcommand{\wh}{\widehat}

    \newcommand{\ov}{\overline}

    \newcommand{\ra}{\rightarrow}

\newcommand{\xto}{\xrightarrow}

    \theoremstyle{plain}

   \newtheorem{thm}{Theorem}[section]\newtheorem{cor}[thm]{Corollary}
    \newtheorem{lem}[thm]{Lemma}  \newtheorem{prop}[thm]{Proposition}

\theoremstyle{remark} 
\theoremstyle{remark} 
\theoremstyle{remark}

    \numberwithin{equation}{section}

\DeclareFontFamily{U}{wncy}{}
\DeclareFontShape{U}{wncy}{m}{n}{<->wncyr10}{}
\DeclareSymbolFont{mcy}{U}{wncy}{m}{n}

\begin{document}

\title[Non-vanishing of central $L$-values of Gross family of curves]{Non-vanishing of central $L$-values of the Gross family of elliptic curves}
\author{Yukako Kezuka} \thanks{The first author is supported by the European Union's Horizon 2020 research and innovation programme under the Marie Sk\l odowska-Curie grant agreement No.\,101026826.} \author{Yong-Xiong Li}

\selectlanguage{english}
\begin{abstract}
We prove non-vanishing theorems for the central values of~$L$-series of quadratic twists of the Gross elliptic curve with complex multiplication by the imaginary quadratic field $\mathbb{Q}(\sqrt{-q})$, where $q$ is any prime congruent to $7$ modulo $8$. This completes the non-vanishing theorems proven by Coates and the second author in which the primes $q$ were taken to be congruent to $7$ modulo $16$. From this, we obtain the finiteness of the Mordell--Weil group and the Tate--Shafarevich group for these curves. For a prime~$\mathfrak{P}$ lying above the prime~$2$, we also prove a converse theorem in the rank $0$ case and the $\mathfrak{P}$-part of the Birch--Swinnerton-Dyer conjecture for the higher-dimensional abelian varieties obtained by restriction of scalars.
\end{abstract}

\maketitle

\section{Introduction}\label{S1}

Let~$q$ be a prime number congruent to $7$ modulo $8$. We define $K$ to be the imaginary quadratic field $\mathbb{Q}(\sqrt{-q})$, viewed as a subfield of~$\mathbb{C}$, with ring of integers~$\mathcal{O}_K$ and class number $h$. Let~$H=K(j(\mathcal{O}_K))$ be the Hilbert class field of~$K$, where $j(\mathcal{O}_K)$ denotes the $j$-invariant of the complex lattice $\mathcal{O}_K$. The prime~$2$ then splits in~$K$, say $2\mathcal{O}_K=\mathfrak{p}\mathfrak{p}^*$, and since $q\equiv 3\bmod 4$ is a prime, $h$ is odd by the genus theory due to Gauss.

Amongst the elliptic curves with complex multiplication by~$K$, Gross has introduced in \cite{Gross78} an elliptic curve $A$ with the particularly favourable properties (in fact, the Gross curve can be defined for any prime~$q\equiv 3\bmod 4$). The Gross elliptic curve $A$ is the unique elliptic curve which is defined over~$\mathbb{Q}(j(\mathcal{O}_K))$ with complex multiplication by~$\mathcal{O}_K$, minimal discriminant $(-q^3)$, and which is a $\mathbb{Q}$-curve, in the sense that the curve $A$ is $H$-isogeneous to all its conjugates $A^\sigma$, where $\sigma$ is any element of the Galois group ${\rm Gal}(H/K)$. In addition, Gross has shown that $A$ has a global minimal model over~$\mathcal{O}_H$, and we write $\Omega_\infty$ for the corresponding complex period.

The study of the arithmetic of~$A$ has attracted many mathematicians since it was developed in \cite{Gross78}. In particular, Gross has shown in \cite{Gross78}  by a $2$-descent argument that $A$ has Mordell--Weil rank zero over~$H$. Motivated by this, Rohrlich showed in \cite{Ro1} that the $L$-function $L(A/H,s)$ of~$A$ over~$H$ does not vanish at $s=1$. Following this, Rodriguez Villegas gave another proof in \cite{rv} for the fact that $L(A/H,1)\neq 0$ using a special value formula which links the $L$-values to the square of twisted by class group character linearly combination of theta values. This formula was parallel to the work of Waldspurger \cite{wal} on general ${\rm GL}_2$ automorphic $L$-functions.

Let~$A^{(D)}$ be a quadratic twist of~$A$ by the extension $H(\sqrt{D})/H$ for a square-free integer $D$. The arithmetic of~$A^{(D)}$ has also received a lot of attention. For example, using theta liftings and analytic methods, Yang \cite{yang1, yang2} showed that $L(A^{(D)}/H,1)\neq 0$ when $D$ is small compared to $q$, roughly, $D<q^{\frac{1}{4}}$.

We now introduce some notation to state our main theorem. Let~$\fR$ denote the set of all square-free positive integers $R$ of the form $R = r_1\cdots r_k$, where $k\geq 0$ ($k=0$ means $R=1$) and $r_1, \ldots, r_k$ are distinct primes such that (i) $r_i \equiv 1 \mod 4$, and  (ii) $r_i$ is inert in~$K$, for~$i=1, \ldots, k$. These conditions are motivated by the theory of 2-descent, which shows that when (i) and (ii) hold, the rank of~$A^{(R)}(H)$ is equal to zero (see \cite{CL} or \cite{choi}).

We write $L(A^{(R)}/H, s)$ for the complex $L$-series of~$A^{(R)}/H$. By Deuring's theorem, $L(A^{(R)}/H, s)$ is a product of Hecke $L$-functions coming from the Hecke character $\psi_R$ associated to $A^{(R)}/H$. The Hecke character $\psi_R$ is of the form $\phi_R\circ N_{H/K}$, where $\phi_R$ is a Hecke character over~$K$ defined in Section~\ref{S2} and $N_{H/K}$ denotes the norm map from $H$ to $K$. Let~$L(\phi_R,s)$ denote the $L$-function of~$\phi_R$, which has analytic continutation to the whole complex $s$-plane.
In view of the relations discussed in Section~\ref{S2}, it suffices to show $L(\overline{\phi}_R,1)\neq 0$ to conclude  $L(A^{(R)}/H, 1) \neq 0$. Let~$T$ be the CM field generated by values of~$\phi_R$ over~$K$, and let~$\fP$ denote the special prime of~$T$ lying above~$2$ whose existence is given in Proposition~\ref{s1}. Without loss of generality, we assume that $\fP$ lies above the prime~$\fp$ of~$K$. We shall prove the following result.

\begin{thm}\label{main}
Let $q$ be a prime congruent to $7$ modulo $8$ and $K = \mathbb{Q}(\sqrt{-q})$. Let $H$ be the Hilbert class field of $K$ and $A$ the Gross curve defined over~$H$ with complex period $\Omega_\infty$.
For $R = r_1 \cdots r_k \in \fR$, denote by $A^{(R)}$ the twist of $A$ by $H(\sqrt{R})/H$. Let $\psi_R = \phi_R \circ N_{H/K}$ be the Hecke character associated to $A^{(R)}/H$, where $\phi_R$ is a Hecke character over~$K$. Let $T$ be the CM field generated over~$K$ by the values of~$\phi_R$ and $TH$ the compositum of the fields $T$ and $H$. Then
$$
\frac{L(\overline{\phi}_R, 1)\sqrt{R}}{\Omega_\infty} \in TH,
$$
and for any prime~$\mathcal{P}$ of~$TH$ lying above~$\fP$, we have
$$
\ord_{\mathcal{P}}\left(\frac{L(\overline{\phi}_R, 1)\sqrt{R}}{\Omega_\infty}\right) = k - 1.
$$
In particular, $L(A^{(R)}/H, 1) \neq 0$, and both the Mordell–Weil group $A^{(R)}(H)$ and the Tate–Shafarevich group $\Sha(A^{(R)}/H)$ are finite.
\end{thm}

As a consequence of Theorem~\ref{main}, we obtain the following density result. Let $D\geq 1$ denote a fundamental discriminant, and let
$$N(X) = \#\{\text{$D < X$ : $L(A^{(D)}/H,1)\neq 0\}$}.$$
Then we have
\begin{cor} \label{cor1.2}
$$N(X)>\!\!>\frac{X}{\log^{\frac{3}{4}} X}.$$
\end{cor}

Combining with the main theorem of \cite{BF}, we obtain the following:

\begin{thm}
Under the same assumptions as in Theorem~\ref{main}, the full the Birch--Swinnerton-Dyer conjecture holds for~$A^{(R)}$ over~$H$.
\end{thm}

For other aspects concerning the order of $\Sha(A^{(R)}/H)$, see also \cite{yli}.

We give several comments on Theorem 1.1.
The implication that $L(A^{(R)}/H, 1) \neq 0$, then $A^{(R)}(H)$ is finite, is a theorem of Coates--Wiles \cite{CW} and its extensions due to Arthaud \cite{A} and Rubin \cite{Ru1}.
Theorem~\ref{main} provides an infinite family of elliptic curves defined over number fields (not over~$\mathbb{Q}$), with conductors divisible by arbitrarily finitely many primes, such that their $L$-values at $s = 1$ are non-vanishing. In fact, we determine the exact $2$-adic valuation of the corresponding algebraic $L$-values. As far as we know, the study of the $2$-adic properties of algebraic $L$-values associated to elliptic curves was initiated by Razar~\cite{razar}, and later extended by Zhao~\cite{Zhao1, Zhao2} to elliptic curves over~$\mathbb{Q}$ with conductors divisible by arbitrarily finitely many primes. For~$2$-adic aspects of the algebraic derivative values of~$L$-series, see Tian's work on the congruent number problem~\cite{ye} (or \cite{ye2}).

The implication that $L(A^{(R)}/H, 1) \neq 0$, then both $A^{(R)}(H)$ and $\Sha(A^{(R)}/H)$ are finite follows from the theorem of Gross--Zagier and Kolyvagin. Recall that $A$ is a $\mathbb{Q}$-curve and its restriction of scalars is modular, and the relation \eqref{GrR} satisfied by the Hecke characters can be used to show that the curves we deal with appear as quotients of the Jacobian $J_0(N)$ (see \cite{BG}).

Letting $R=1$ in Theorem~\ref{main}, we obtain~$L(A/H,1)\neq0$, reproving the result of Rohrlich in \cite{Ro1}. The key difference is that our result gives the exact valuation at the primes above $\fP$ of the algebraic $L$-value for the base curve, \textit{which provides the crucial initial step in the induction on which the proof of Theorem~\ref{main} relies}.

We remark that when $q\equiv 7\mod 16$, the above result was proven by Coates and the second author in \cite{CL}. The new part of the theorem thus concerns the case when $q\equiv 15\bmod 16$. We stress that this case could not be covered in \cite{CL}, and this is due to the fact that the corresponding Iwasawa module $X(F_\infty)$ introduced in Section~\ref{S5} is trivial in the case $q\equiv 7\mod 16$, but it was proven to be always non-trivial in the case $q\equiv 15\bmod 16$ (see \cite{JLi}). It is thus essential to obtain a better understanding of this Iwasawa module.

The module $X(F_\infty)$ is the Pontryagin dual of a certain Selmer group (see Theorem~\ref{sel-rel-x}). A key method used to prove Theorem~\ref{main} is thus a combination of various descents at finite levels and an infinite descent in the setting of Iwasawa theory.

In proving Theorem~\ref{main}, we will first study the abelian variety $B=\Res_{H/K}(A)$ over~$K$. Another key ingredient is the Iwasawa main conjecture at the special prime~$\fP$, following \cite{CL}.  Indeed, the field $F_\infty$ above is given by~$K(B_{\fP^\infty})$, obtained by adjoining to $K$ all $\fP$-power division points on $B$. In order to treat the case $p=2$, we also need to show that $X(F_\infty)$ is a finitely generated $\BZ_2$-module. This will follow from our earlier work \cite{CKL}, in which we show that the Iwasawa module $X(H(A_{\fp^\infty}))$ corresponding to the field $H(A_{\fp^\infty})$ is a finitely generated $\BZ_2$-module. Combined with the descent arguments, this gives us a non-vanishing result for the base curve $A$ and a rank zero $\mathfrak{P}$-converse theorem for~$B$. Furthermore, we prove the $\mathfrak{P}$-part of the Birch--Swinnerton-Dyer conjecture for~$B$ over~$K$, a refinement due to Buhler and Gross \cite{BG}. Using this as the base case, we apply an induction argument on the number of prime divisors of~$R\in \mathfrak{R}$, a generalisation of Zhao's method to abelian varieties, to extend the non-vanishing result to the twists $A^{(R)}$ as in Theorem~\ref{main}.

We note that there have recently been several developments related to the subject of our paper, as discussed in \cite{BF, BT} (see also \cite[Thm.\,4.4]{BT2} and \cite[Thm.\,22]{ye2}). We remark that the standard $2$-converse theorem is not directly applicable in our setting. Indeed, the numerical examples in \cite{DJS} show that, in particular for $q = 431$ and $q = 751$, the $2$-part of the Tate--Shafarevich group $\Sha(A/H)$ is non-trivial. Since a 2-descent on $A/H$ yields information only about the $2$-Selmer group, understanding the $\BZ_2$-corank of the full $2^\infty$-Selmer group requires a higher 2-descent, i.e., a study of the $4$- or $8$-Selmer group. This process is notably difficult for $A$. Another complication arises from the primes $\fP_T$ of $T$ lying above $\fp$. Since $T$ is not Galois over~$K$, the $\fP_T$-components of the Selmer group of $B/K$ are not symmetric, making it difficult to determine which $\fP_T$-part vanishes. A key observation in our case is to identify a special prime~$\fP$ of $T$ such that the $\fP^\infty$-part of the Selmer group of $B/K$ vanishes. Moreover, our proofs of both the $\fP$-converse theorem and the refined $\fP$-part of the Birch--Swinnerton-Dyer conjecture are direct and self-contained, and do not rely on various formulations of the Iwasawa main conjecture.

This paper is structured as follows. In Section~\ref{S2}, we introduce some notation and preliminary results. We give a $2$-descent argument in Section~\ref{S3} showing that~$B$ has no first descent at the special prime~$\fP$ of~$T$ lying above~$2$. In Sections~4 and~5, we prove a rank zero converse theorem for~$B$ at the prime~$\fP$, and prove the $\fP$-part of the Birch--Swinnerton-Dyer conjecture. This uses Iwasawa theory for the prime~$\mathfrak{P}$ and the descent results, giving the exact $\fP$-adic valuation of the algebraic part of the Hecke $L$-value associated to $B/K$ at $s=1$. In Section~\ref{S6}, we apply a generalisation of Zhao's induction method to obtain the non-vanishing result for the twisted curves, concluding the proof of Theorem~\ref{main}. The proof of Corollary~\ref{cor1.2} can be found at the end of Section~\ref{S6}.\\

\textbf{Acknowledgement}\quad
We would like to express our gratitude to Pierre Colmez for his insightful comments on this paper and for suggesting the inclusion of the density argument presented in Corollary 1.2. Our thanks also extend to Shou-Wu Zhang for initiating the density question at the Iwasawa 2023 conference. Additionally, we would like to thank the anonymous referees for their constructive feedback, which led to significant improvements in the organisation and presentation of the paper.
We wish to dedicate this paper to commemorate John Coates. John encouraged us to consider the problems in this paper. He believed in the power of small primes in number theory, and he inspired us to pursue it. With this paper, we wish to tell him that what he predicted is \mbox{indeed true}.

\section{Preliminaries}\label{S2}

Since the minimal discriminant ideal $(-q^3)$ of~$A$ is a principal ideal, it raises the question as to whether there exists a global minimal Weierstrass equation for~$A$ over the ring of integers $\mathcal{O}_H$ of~$H$. Gross has proven that this is indeed the case~\cite{Gross82}, and we fix one such equation
\begin{equation}\label{1}
  y^2+a_1xy+a_3y=x^3+a_2x^2+a_4x+a_6,
\end{equation}
with $a_i\in \mathcal{O}_H$.  We let~$B/K$ be the abelian variety which is the restriction of scalars from $H$ to $K$ of the Gross curve $A/H$. For each $R \in \fR$, we write $B^{(R)}$ for the twist of~$B$ by the quadratic extension $K(\sqrt{R})/K$, and $A^{(R)}$ for the twist of~$A$ by the quadratic extension $H(\sqrt{R})/H$. From \cite{BG} or \cite{Gross78}, we know that $B^{(R)}$ is in fact the restriction of scalars from $H$ to $K$ of~$A^{(R)}$.

We write
$$
\mst=\End_K(B^{(R)})=\End_K(B)$$
and
$$T = \mst\otimes_\BZ \BQ.
$$
Then $T$ is a CM field of degree $h$ over~$K$, where $h$ denotes the class number of~$K$. Let~$\alpha: T\hookrightarrow \BC$ be an embedding of~$T$ into~$\BC$ which extends our embedding of~$K$ into~$\BC$. There are $h$ such choices of~$\alpha$. We write $\psi$ for the Hecke character of~$A/H$ and $\phi=\phi^\alpha$ for the Hecke character of~$B/K$ relative to $\alpha$. We then have $\psi = \phi \circ N_{H/K}$, where $N_{H/K}$ denotes the norm map from $H$ to $K$. Since $(R, q) = 1$, the Hecke character  $\phi_R$ of~$B^{(R)}/K$ is then given by~$\phi_R=\phi\chi_R$, where $\chi_R$ denotes the abelian character of~$K$ defining the quadratic extension $K(\sqrt{R})/K$. Moreover, by the explicit description of the Hecke characters $\psi_R$ and $\phi_R$ in \cite{BG}, we have
\begin{equation}\label{GrR}
\psi_R=\phi_R\circ N_{H/K}.
\end{equation}
Note that the CM field $T$ is also generated by the values of the Hecke character $\phi$ (or $\phi_R$) over~$K$.
In what follows, we shall write $\ov{\psi}$ and $\ov{\phi}$ for the complex conjugate characters of~$\psi$ and $\phi$, respectively.

Since $A$ is a $\BQ$-curve, we have $L(A/H,s)=L(\ov{\psi},s)^2$ by a theorem of Deuring, and we have the relation between $L$-series \cite[Section~18]{Gross78}
\begin{equation}\label{phipsi}L(\ov{\psi},s)=\prod_{\alpha\in \Hom_K(T,\BC)}L(\ov{\phi}^\alpha,s),
\end{equation}
where the product runs over the $h$ distinct embeddings of~$T$ lying above our fixed embedding of~$K$. We will first show the non-vanishing of~$L(\ov{\phi},1)$. Since the same arguments show that $L(\ov{\phi}^\alpha,1)\neq 0$ for any $\alpha$, we can conclude that $L(A/H,1)\neq 0$. Similarly, in order to show $L(A^{(R)}/H,1)\neq 0$, it is sufficient to show $L(\ov{\phi}_R,1)\neq 0$.

We now choose a prime of~$T$ lying above~$2$, which is well suited for our purpose. From Section~13 of \cite{Gross78} (or \cite{BG}), the index of~$\mst$ in the maximal order of~$T$ is prime to 2.  As $q\equiv 7\mod 8$, the prime~$2$ splits in~$K$ into two distinct primes, which we will denote by~$\fp$ and $\fp^*$. The following lemma, whose detailed proof can be found in~\cite{BG} or Lemma 2.1 in \cite{CL}, gives the existence of a degree one prime~$\fP$ of~$T$ above~$\fp$. The reason for the existence of such a prime comes from considering the action of~$\mst$ on $B(K)=A(H)$ and using a fundamental theorem of Coates and Wiles (see the proof in Lemma~\ref{lem3-1}).
We remark this is the special prime mentioned in the introduction, which will play a fundamental role throughout this paper.

\begin{prop}\label{s1}
There exists a unique, unramified degree one prime~$\fP$ of~$T$ lying above~$\fp$.
\end{prop}

\noindent Of course, since the index of~$\mst$ in the ring of integers of~$T$ is prime to 2, $\fP \cap \mst$ will be a degree one prime ideal of~$\mst$, which for simplicity we shall again denote by~$\fP$. We will also use the degree one prime~$\fP^*$ of~$T$ lying above~$\fp^*$, which is obtained by applying the complex conjugation to $\fP$.

\section{The first $2$-descent on $B$ over~$K$}\label{S3}

Since $\mst$ is the ring of~$K$-endomorphisms of the abelian variety $B$, for each integer $n\geq1$, we define $B_{\fP^n} = \bigcap_{x \in \fP^n} B_x$, where $B_x$ denotes the group of~$x$-division points in~$B(\overline{K})$ for $x \in \mst$. Let~$K(B_{\fP^n})$ be the field generated by these points over~$K$, and set $F_\infty = \bigcup_{n \geq 1} K(B_{\fP^n})$. Iwasawa theory of~$B$ for the prime~$\mathfrak{P}$ shall deal with the extension $F_\infty/K$. Its Galois group is of the form $\Gamma\times \Delta$ where $\Gamma\simeq \BZ_2$ and $\Delta\simeq \BZ/2\BZ$. The main conjecture of Iwasawa theory
relates the corresponding Iwasawa module~$X(F_\infty)$ to a $\mathfrak{P}$-adic $L$-function, which interpolates the Hecke $L$-value $L(\ov{\phi},1)$ (see Section~5). As we remarked earlier or from \cite{CL}, the Pontryagin dual of~$X(F_\infty)$ is the $\fP^\infty$-Selmer group $\Sel_{\fP^\infty}(B/F_\infty)$.
We shall prove that its $\Gamma$-invariant part $\Sel_{\fP^\infty}(B/F_\infty)^\Gamma$ is finite and compute the $\fP$-adic valuation of its order to relate it to the $\fP$-adic valuation of the Hecke $L$-value. This requires carrying out various descents, both at finite and inifinite levels, and proving exact relations between them. In this section, we will carry out a $2$-descent on $B$ over~$K$.

Let~$G=\Gal(H/K)$, which is a group of order $h$, an odd number by our choice of~$q$. In this section, we will use a $G$-invariant $2$-descent result on $A/H$ due to Gross (see the proof of Proposition~\ref{30}) to prove a $2$-descent result on $B/K$.
For each integer $n\geq1$, let~$\Sel_{\fP^n}(B/K)$ denote the $\fP^n$-Selmer group of~$B$ over~$K$.
Recall that $\fP$ lies above~$\fp$, where $\fp$ is one of the primes of~$K$ lying above~$2$.
For each integer $n\geq1$, let~$A_{\fp^n}$ be the Galois module of~$\fp^n$-division points on $A$. We similarly denote the $\fp$-Selmer group $\Sel_{\fp}(A/H)$ of~$A$ over~$H$. The Selmer group $\Sel_{\fp}(A/H)$ is a $G$-module, and we denote by~$\Sel_{\fp}(A/H)^G$ the subgroup of~$\Sel_{\fp}(A/H)$ consisting of elements fixed by~$G$.

\begin{prop}\label{30}
We have
\[\Sel_{\fp}(A/H)^G\simeq \BZ/2\BZ.\]
\end{prop}

\begin{proof}
Let~$\Sel_{2}(A/H)$ be the $2$-Selmer group of~$A/H$ using the Galois module $A_2$ of~$2$-torsion points on $A$. This is also a $G$-module, and we denote by~$\Sel_{2}(A/H)^G$ its Galois invariant subgroup. The $G$-invariant Selmer group is much easier to compute than the full Selmer group, and Gross shows in \cite{Gross78} that
\begin{equation}\label{g-f}
  \Sel_2(A/H)^G\simeq \CO_K/2\CO_K.
\end{equation}
For any integral ideal $\fx$ in~$\CO_K$, we let~$H(A_\fx)$ the field generated by the $\fx$-division points over~$H$.
Since $A$ has bad reduction only at primes of~$H$ above~$q$, and a fundamental result of Coates--Wiles (see~\cite{CL} or~\cite{CW}) shows that $A$ acquires good reduction everywhere over~$H(A_{\fp^2})$ (or over~$H(A_{\fp^{*2}})$), it follows that $A_{2^\infty}(H)= A_2$.

Therefore, we have the decomposition
$$A_2=A_\fp\oplus A_{\fp^*}$$
as $\Gal(\ov{K}/H)$-modules. By checking the local conditions at each prime, this implies
$$\Sel_2(A/H)=\Sel_{\fp}(A/H)\oplus\Sel_{\fp^*}(A/H),$$
where $\Sel_{\fp^*}(A/H)$ is the $\fp^*$-Selmer group of~$A$ over~$H$.
Since both $\Sel_\fp(A/H)$ and $\Sel_{\fp^*}(A/H)$ are $G$-modules that are finite of even order, and since $G$ has odd order (by Gauss’s genus theory, as $h$ is odd), we obtain
\[\Sel_2(A/H)^G=\Sel_\fp(A/H)^G\oplus\Sel_{\fp^*}(A/H)^G.\]
Note that $\Sel_{\fp}(A/H)^G$ is an $\CO_{K_\fp}$-module where $\CO_{K_\fp}$ is the ring of integers of~$K_\fp$, and $\CO_{K_\fp}/\fp\CO_{K_\fp}\simeq\BZ/2\BZ$. The proposition then follows by taking the tensor product on both sides of \eqref{g-f} with $\CO_{K_\fp}$.
\end{proof}

We will also use the following isomorphism of Galois modules.

\begin{lem}\label{lem3-1}
There exists an isomorphism  of~$\Gal(\ov{K}/H)$-modules
\[B_\fP\simeq A_\fp.\]
\end{lem}

\begin{proof}
Since $B_\fP$ is a $\Gal(\ov{K}/K)$-module, we may naturally regard it as a $\Gal(\ov{K}/H)$-module. By the definition of the Weil restriction, we have an identification $B(K) = A(H)$. As shown in the proof of Proposition~\ref{30}, we have
$$B_{2^\infty}(K)=A_{2^\infty}(H)=A_2=A_\fp\oplus A_{\fp^*}.$$
From \cite{BG} (see also \cite{CL}), the $\mst$ action on $B_{2^\infty}(K)$ factors through the ideal $\fP\fP^*$. It follows that both $B_\fP$ and $A_\fp$ are isomorphic to $\BZ/2\BZ$, with trivial $\Gal(\ov{K}/H)$-action.
The lemma now follows.
\end{proof}

\begin{prop}\label{2-descent-B}
We have
\[\Sel_\fP(B/K)\simeq B_\fP.\]
In particular, $B(K)$ is finite and
\[\Sha(B/K)(\fP)=0.\]
Here, we denote by~$\Sha(B/K)(\fP)$ the $\fP$-primary part of the Tate--Shafarevich group of~$B/K$.
\end{prop}

\begin{proof}
The proof crucially depends on the $2$-descent result in Proposition~\ref{30}. We recall the following definitions for~$\Sel_\fP(B/K)$ and $\Sel_\fp(A/H)$ using the exact sequences
\[0\ra\Sel_\fP(B/K)\ra H^1(K,B_\fP)\ra\prod_v \frac{H^1(K_v, B_{\fP})}{{\rm Im}\, \kappa_v(B)}\]
and
\[0\ra \Sel_\fp(A/H)\ra H^1(H,A_\fp)\ra \prod_w \frac{H^1(H_w,A_\fp)}{{\rm Im}\, \kappa_w(A)},\]
where, $\kappa_v(B)$ (resp. $\kappa_w(A)$) denotes the local Kummer map of~$B$ (resp. $A$) at~$v$ (resp. $w$).
Recall that $G=\Gal(H/K)$. By taking the $G$-invariant part of the second exact sequence, we obtain the exact sequence
\[0\ra\Sel_\fp(A/H)^G\ra H^1(H,A_\fp)^G\ra \left(\prod_w \frac{H^1(H_w,A_\fp)}{\Im\kappa_w(A)}\right)^G.\]
Using Lemma~\ref{lem3-1} and noting that  $G$ is of odd order, we see that the restriction map
\begin{equation}\label{global-iso}
  i:H^1(K,B_\fP)\ra H^1(H,A_\fp)^G
\end{equation}
is an isomorphism. Furthermore, locally at each prime~$v$ of~$K$ we can identify~$G$ with the group \mbox{$\Gal((K_v\otimes_K H)/K_v)$}. Thus, the restriction map also gives the isomorphism
\begin{equation}\label{local-iso}
  j_0: H^1(K_v,B_\fP)\ra\left(\prod_{w\mid v}H^1(H_w,A_\fp)\right)^G.
\end{equation}
Now, since $B$ is the restriction of scalars of~$A$ from $H$ to $K$ (or from \cite{Gross78}), we have
\begin{align*}
 B(K_v)=A(K_v\otimes_K H)=\prod_{w\mid v}A(H_w).
\end{align*}
Since $B(K_v)$ is fixed by~$G=\Gal((K_v\otimes_K H)/K_v)$, given any point $(P_w) \in \prod_w A(H_w)$ and any $\sigma\in G$, we may identify $(P_w)$ with $(P^\sigma_{\sigma w})$.

Note that we have $\CO_{T_\fP}/\fP=\CO_K/\fp=\BZ/2\BZ$, and thus
\[\left(\prod_{w\mid v}A(H_w)\right)\otimes (\CO_K/\fp)=B(K_v)\otimes(\CO_{T_\fP}/\fP).\]
Consider the maps
\[B(K_v)\otimes(\CO_{T_\fP}/\fP)\stackrel{\kappa_v(B)}{\hookrightarrow}H^1(K_v,B_\fP)
\stackrel{\simeq}{\longrightarrow}\left(\prod_{w\mid v}H^1(H_w,A_\fp)\right)^G,\]
where the last isomorphism is given by~$j_0$ in \eqref{local-iso}. From our remark on the Galois action of~$G$, we see that the composition of these maps coincides with $\prod_{w\mid v}\kappa_w(A)$. It then follows from the definition of the Selmer groups and the isomorphisms \eqref{global-iso} and \eqref{local-iso} that $i$ sends $\Sel_\fP(B/K)$ isomorphically to $\Sel_\fp(A/H)^G$. Noting the descent exact sequence
\[0\to B(K)\otimes (\CO_{T_\fP}/\fP)\to \Sel_\fP(B/K)\to \Sha(B/K)_\fP\to 0,\]
where $\Sha(B/K)_\fP$ is the $\fP$-division part of~$\Sha(B/K)$,
the result now follows from Proposition~\ref{30} and Lemma~\ref{lem3-1}.
\end{proof}

As a consequence, we obtain:

\begin{cor}\label{sel-p-inf}
We have
\[\Sel_{\fP^\infty}(B/K)=0.\]
In particular, by replacing $\fP$ with $\fP^*$, we have
\[\Sel_{\fP^{*\infty}}(B/K)=0,\]
and both $B(K)$ and $\Sha(B/K)(\fP\fP^*)$ are finite.
\end{cor}

\begin{proof}
Consider the $\fP^\infty$-descent exact sequence:
\[0\to B(K)\otimes(T_\fP/\CO_{T_\fP})\to \Sel_{\fP^\infty}(B/K)\to \Sha(B/K)(\fP)\to 0.\]
By Proposition~\ref{2-descent-B}, the finiteness of~$B(K)$ implies that
$$B(K)\otimes(T_\fP/\CO_{T_\fP})=0.$$
Moreover, we have $\Sha(B/K)(\fP) = 0$. Therefore,
\[\Sel_{\fP^\infty}(B/K)=0.\]
The remaining claims follow by applying complex conjugation to the above equality.
\end{proof}

\section{Descent on $B$ over the fields $K$ and $F = K(B_{\fP^2})$}\label{S4}

The elliptic curve $A$ only has additive reduction at each place $w$ of~$H$ lying above~$\fq=\sqrt{-q}\CO_K$. For each such $w$, we write $A_0(H_w)$ for the subgroup of~$A(H_w)$ consisting of points with non-singular reduction modulo $w$, and put $\fC_w = A(H_w)/A_0(H_w)$. Of course, $\fC_w$ is a $\CO_K$-module, and by using N\'{e}ron's enumeration of the reduction types, the order of~$\fC_w$ is at most $4$ and $\fC_w$ has exponent dividing $2$ (see \cite[Proposition 4.5]{gross-bsd}). For a finite set $\fW$, we denote by $\#(\fW)$ the cardinality of $\fW$.

\begin{lem}\label{ne} Let~$w$ be a place of~$H$ lying above~$\fq=\sqrt{-q}\CO_K$. Then $\fC_w \simeq \CO_K/2\CO_K$ as a $\CO_K$-module. In particular, $\fC_w$ is of order 4.
\end{lem}

\begin{proof} Since $w \nmid 2$ and $A$ has additive reduction at $w$, multiplication by~$2$ is an automorphism of~$A_0(H_w)$. As $\fC_w = A(H_w)_2$, it follows that $\fC_w \simeq \CO_K / 2\CO_K$, as claimed.
\end{proof}

Since $B=\Res_{H/K}A$, we have
\begin{equation}\label{ba-1}
  B(K_\fq)=\prod_{w\mid\fq}A(H_w).
\end{equation}
We define
\begin{equation}\label{ba-2}
  B_0(K_\fq)=\prod_{w\mid\fq}A_0(H_w),
\end{equation}
where the two products are taken over all places $w$ of~$H$ above~$\fq$, and we set \mbox{$C_\fq=\frac{B(K_\fq)}{B_0(K_\fq)}$}.

From the restriction of scalars for the N\'{e}ron models associated to $A/H$ and $B/K$, we obtain the following:
\begin{lem}\label{tam}
\noindent
\begin{enumerate}
  \item $C_\fq=\prod_{w\mid\fq}\fC_w$, and
  \item $C_\fq$ is a module over~$\CO_{T}\otimes_\BZ\BZ_2$, and $(C_\fq)(\fP^*)\simeq \CO_{T_{\fP^*}}/\fP^*\CO_{T_{\fP^*}}$.
\end{enumerate}
\end{lem}

\begin{lem}\label{t2.1}\noindent
\begin{enumerate}[(i)]
\item[(1)]  $H^1(K_\fq, B)(\fP)$ is finite of order $2$, and
\item[(2)] $H^1(K_\fp, B)(\fP)$ is finite of order equal to $\mid(1-\phi(\fp)/2)\mid_\fP^{-1}$.
\end{enumerate}
Here, $\mid\cdot\mid_\fP$ is the multiplicative valuation on $T_\fP$, normalised so that $|2|_\fP=2^{-1}$.
\end{lem}

\begin{proof}
The proof uses Tate local duality.
Note that $B$ is a self-dual abelian variety (see~\cite{gross-bsd}).
Since Lemma~\ref{tam} implies that $(C_\fq)(\fP^*)$ is annihilated by~$\fP^*$, and $B$ has additive reduction at $\fq$,
Tate local duality shows that the Pontryagin dual of~$H^1(K_\fq, B)(\fP)$ is isomorphic to $B_{\fP^*}(K_\fq)$. This proves (1). For  (2), Let~$\kappa(K_\fp)$ be the residue field of~$K_\fp$. Since $B$ has good reduction at $\fp$, we Let~$\widetilde{B}$ be the reduction abelian variety of~$B$ modulo $\fp$. As we have a global minimal model of~$A$ over~$H$,
we can reduce $B(K_\fp)$ using the relation \eqref{ba-1}. By Tate local duality, the Pontryagin dual of~$H^1(K_\fp,B)(\fP)$ is equal to $\widetilde{B}(\kappa(K_\fp))(\fP^*)$.
By the theory of complex multiplication, $\phi(\fp)$ is the unique element of the ring of endomorphisms $\mst = \End_K(B)$ whose reduction modulo $\fp$ is the Frobenius endomorphism of~$\widetilde{B}$ over~$\kappa(K_\fp)$. Noting that $\phi(\fp)\ov{\phi}(\fp)=2$ and $\phi(\fp)-1$ is a $\fP$-unit, we obtain
$$
\left|\#(\tilde{B}(k(K_\fp))(\fP^*))\right|^{-1}_\fP = \mid(\ov{\phi}(\fp) - 1)(\phi(\fp)-1)\mid^{-1}_{\fP}=\left|\left(1-\frac{\phi(\fp)}{2}\right)\right|^{-1}_\fP.
$$
This completes the proof of (2).
\end{proof}

Let~$L$ be an algebraic extension over~$K$, and let~$S$ be a finite set of primes of~$K$. We define
\[\Sel^S_{\fP^\infty}(B/L)=\ker\left(H^1(L, B_{\fP^\infty})\ra \prod_{v\nmid S}H^1(L_v, B)\right),\]
where the product runs over all primes $v$ of~$L$ not lying above the primes in~$S$ and $L_v$ denotes the compositum field of all completions at~$v$ of finite extensions over~$K$ contained in~$L$. In the following, we shall simply write $\Sel'_{\fP^\infty}(B/L)$ for~$\Sel^{\{\fp\}}_{\fP^\infty}(B/L)$ and $\Sel^\CW_{\fP^\infty}(B/L)$ for~$\Sel^{\{\fp, \fq\}}_{\fP^\infty}(B/L)$.

\begin{prop}\label{t2.2} We have
\begin{align*}
\left|\#\left(\Sel_{\fP^\infty}^\CW(B/K)\right)\right|^{-1}_\fP =\left|\left(1-\frac{\phi(\fp)}{2}\right)\right|_\fP^{-1}.
\end{align*}
\end{prop}

\begin{proof} Let~$\pi$ be a fixed element of~$\mst$ satisfying the factorisation $(\pi)=\fP^r$ for some integer $r\geq 1$. For each integer $n \geq 1$, we have the short exact sequence
\begin{equation}\label{2.5}
0 \ra \Sel_{\pi^n}(B/K) \ra \Sel_{\pi^n}^\CW(B/K)\xto{u_n} \prod_{v\in\CW}H^1(K_v,B)_{\pi^n},
\end{equation}
where $\Sel^\CW_{\pi^n}(B/K)$ is defined in a similar manner for~$B_{\pi^n}$, and we write $u_n$ for the right hand homomorphism in this sequence. On the other hand, by the definition of the Selmer group $\Sel_{\pi^{*n}}(B/K)$, we have the natural homomorphism
\begin{equation}\label{2.6}
s_n:\Sel_{\pi^{*n}}(B/K)\ra \prod_{v\in\CW}B(K_v)/\pi^{*n}B(K_v).
\end{equation}
Note that the groups on the right of \eqref{2.5} and \eqref{2.6} are dual to each other by Tate local duality and self-dualness of~$B$. By the modified Poitou--Tate sequence  (see, for example, \cite[Section~3.3]{bas}), we conclude that
$\Coker(u_n)$ is equal to the Pontryagin dual of~$\Im(s_n)$ for all $n \geq 1$. Hence, writing $u_\infty$ for the inductive limit of the maps $u_n$ as $n \to \infty$, it follows that
$\Coker(u_\infty)$ is dual to the image of the map $\fs_\infty = \varprojlim_{n}s_n$, where
\begin{align*}\label{2.7}
\fs_\infty:\varprojlim_{n}\Sel_{\pi^{*n}}(B/K)\ra \prod_{v\in\CW}B(K_v)\otimes_{\mst}\CO_{T_{\fP^*}}.
\end{align*}
From Corollary~\ref{sel-p-inf}, both $B(K)$ and $\Sha(B/K)(\fP^*)$ are finite. It follows that $\varprojlim_{n}\Sel_{{\pi^*}^n}(B/K)$ is isormorphic to $B(K)(\fP^*)$, which is cyclic of order 2.
Moreover, $\fs_\infty$ is injective because $\CW$ is non-empty. Thus, the cokernel of~$u_\infty$ has order equal to $\#(B(K)(\fP^*))$. The result now follows on taking the inductive limit over all $n \geq 1$ of the exact sequences \eqref{2.5} and combining with Lemma~\ref{t2.1} and Corollary~\ref{sel-p-inf}.
\end{proof}

Let~$F=K(B_{\fP^2})$. One can easily show that the degree of~$F$ over~$K$ is equal to~$2$. The next proposition is Theorem 2.4 of  \cite{CL}.

\begin{prop}\label{good-red}
The abelian variety $B$ has good reduction everywhere over~$F$.
\end{prop}

Let~$K_\infty$ be the unique $\BZ_2$-extension over~$K$ which is unramified outside of the prime~$\fp$. Then we have $$F_\infty=K(B_{\fP^\infty})=K_\infty(B_{\fP^2}).$$
Let us denote
\[\msg=\Gal(F_\infty/K)=\Gamma\times\Delta,\]
where $\Gamma\simeq \BZ_2$ and $\Delta=\Gal(F_\infty/K_\infty)\simeq\Gal(F/K)$ is of order $2$. We write $\delta$ for the generator of~$\Delta$, so that $\delta$ acts on $B_{\fP^\infty}$ by multiplication by~$-1$. We refer the reader to \cite{CL} for more details on these fields and Galois groups.

Since $\Sel'_{\fP^\infty}(B/F)$ is a module over~$\Gal(F/K)$, we denote by~$\Sel'_{\fP^\infty}(B/F)^\Delta$ the subgroup of elements in~$\Sel'_{\fP^\infty}(B/F)$ fixed by~$\Delta$.

\begin{prop}\label{h-seq}
We have the exact sequence
\begin{equation}\label{seq-n}
  0\ra H^1(F/K, B(F)_{\fP^\infty})\ra \Sel^\CW_{\fP^\infty}(B/K)\ra \Sel'_{\fP^\infty}(B/F)^\Delta\ra 0.
\end{equation}

\end{prop}

\begin{proof}
The assertion, except for the surjectivity, will follow from the commutative diagram
\begin{equation}\label{tu-001}
  \xymatrix{
  0  \ar[r]^{} & \Sel^\CW_{\fP^\infty}(B/K) \ar[d]_{j_1} \ar[r]^{} & H^1(K, B_{\fP^\infty}) \ar[d]_{j_2} \ar[r]^{} & \prod_{v\notin \CW}H^1(K_{v}, B)(\fP) \ar[d]^{j_3} \\
  0 \ar[r]^{} & \Sel'_{\fP^\infty}(B/F)^\Delta  \ar[r]^{} & (H^1(F, B_{\fP^\infty}))^\Delta \ar[r]^{} & \left(\prod_{w\notin \{\fp\}}H^1(F_{w}, B)(\fP)  \right)^\Delta\text{,}}
\end{equation}
where $j_2$ is the restriction map and $j_1$ is induced by the restriction of~$j_2$ to $ \Sel^\CW_{\fP^\infty}(B/K)$. By a theorem of Mazur, $j_3$ is injective, so we get
\[{\rm ker}j_1\simeq {\rm ker}j_2=H^1(F/K, B(F)_{\fP^\infty}).\]
From Proposition~\ref{good-red}, $\fq$ is totally ramified in~$F$ and we denote by~$w$ the unique prime of~$F$ above~$\fq$.
Applying the Tate local duality to the local restriction map
\[r_\fq:H^1(K_\fq, B)(\fP) \to H^1(F_w, B)(\fP),\]
we can easily show that $r_\fq=0$. Therefore, we get
\[ j_2\left(\Sel_{\fP^\infty}^\CW(B/K)\right)\subset (\Sel_{\fP^\infty}'(B/F))^\Delta.\]
By applying the snake lemma to the diagram \eqref{tu-001}, we obtain the exact sequence
\begin{align*}
0\ra H^1(F/K, B(F)_{\fP^\infty})\ra \Sel^\CW_{\fP^\infty}(B/K)\stackrel{j_1}{\longrightarrow}\Sel'_{\fP^\infty}(B/F)^\Delta.
\end{align*}
To complete the proof, we remain to proving that $j_1$ is surjective. The proof goes similarly as \cite[Proposition 3.8]{gon-bsd}, since there are subtle things in the proof, we write a detailed proof below.
We simply write $v=\fq$ in~$\CW$, and let~$I_v$ be the inertia subgroup of some fixed prime of the algebraic closure of~$K$ above~$v$. Since~$v$ is ramified in~$F/K$, we can find
an element $\theta$ of~$I_v$ whose restriction to $F$ is $\delta\big|_{F}$. We now take~$\xi$ to be any element of~$\Sel_{\fP^\infty}'(B/F)^\Delta$. We must show that there
exists a cohomology class $\rho$ in~$H^1(K, B_{\fP^\infty})$ such that $j_2(\rho) = \xi$. Note that any such $\rho$ must automatically lie in~$\Sel^\CW_{\fP^\infty}(B/K)$ by the injectivity of~$j_3$. Since $B$ has good reduction everywhere over~$F$, we can choose a cocycle
representative $g$ of~$\xi$ which is trivial on $I_v\cap \Gal(\ov{K}/F)$, whence $g(\theta^{2n}) = 0$ for all integers $n$, as~$\theta^{2n}$ lies in~$I_v\cap \Gal(\ov{K}/F)$.
For each $z$ in~$B_{\fP^\infty}$, Let~$d(z)$ denote the 1-coboundary on~$\Gal(\ov{K}/F)$ defined by~$d(z)(\sigma) = (\sigma -1)z$. Note that
$$
 (\delta d)(z) = \theta(\theta^{-1}\sigma\theta-1)z = (\sigma-1)\delta(z) = -d(z).
$$
Now, since the cohomology class $\xi$ is fixed by~$\Delta$, we have $(1-\delta)g = d(z)$ for some $z \in B_{\fP^\infty}$. Let~$u$ be any element
of~$B_{\fP^\infty}$ such that $2u = z$. We claim that the equivalent cocycle $f$ defined by~$f = g - d(u)$ is then actually invariant under the
action of~$\Delta$. Indeed, we have
$$
 (1-\delta)f = (1-\delta)g - (1-\delta)d(u) = d(z) - 2d(u) = 0.
$$
Note also that we have $f(\theta^{2n})=0$ for all integers $n$. Now, every element $\tau$ in~$\Gal(\ov{K}/K)$ can be written in the form $\tau = \sigma\theta^i$, where $\sigma$ is in~$\Gal(\ov{K}/F)$ and $i\in \{0, 1\}$. We define the map $h: \Gal(\ov{K}/K) \to B_{\fP^\infty}$ by~$h(\tau) = f(\sigma)$.  We claim that $h$ is indeed a 1-cocycle, that is, taking $\tau_k = \sigma_k\theta^{i_k}$ for~$k\in \{1,2\}$, we must show that
\begin{equation}\label{16}
 h(\tau_1\tau_2) = h(\tau_1) + \tau_1h(\tau_2).
\end{equation}
Note first that, for any $\sigma$ in~$\Gal(\ov{K}/F)$ and any integer $m$, we have $h(\sigma\theta^m) = f(\sigma)$ since $f$ vanishes on the even powers of~$\theta$.
Now, we have $\tau_1\tau_2 = \sigma_1\sigma_2' \theta^{i_1+i_2}$, where $\sigma_2' = \theta^{i_1}\sigma_2\theta^{-i_1}$, whence it follows that
$$
 h(\tau_1\tau_2)=f(\sigma_1\sigma_2' ) = f(\sigma_1) + \sigma_1f(\sigma_2').
$$
But by construction, the cocycle $f$ is fixed by~$\Delta$, and so we have $f(\sigma_2') = \theta^{i_1}f(\sigma_2)$, and the equality \eqref{16} follows.
\end{proof}

Combining Propositions \ref{t2.2} and \ref{h-seq}, and noting that
$$B(K)(\fP^*)\simeq \BZ/2\BZ \quad \text{and} \quad H^1(F/K, B(F)(\fP))\simeq B_{\fP^2}/B_\fP,$$
we obtain the following corollary.

\begin{cor}\label{sel'-f}
The Selmer group $\Sel'_{\fP^\infty}(B/F)^\Delta$ is finite, and we have
\[
\left|\#\left(\Sel'_{\fP^\infty}(B/F)^\Delta\right)\right|^{-1}_\fP =\frac{1}{2}  \left|\left(1-\frac{\phi(\fp)}{2}\right)\right|_\fP^{-1}.
\]
\end{cor}

\section{Infinite descent over the field $F_\infty=K(B_{\fP^\infty})$ and a main conjecture for~$F_\infty/K$}\label{S5}
In this section, we will use an infinite descent method due to Coates \cite{Coa} and an Iwasawa main conjecture to show that
$$L(\ov{\phi},1)\neq0,$$
and to give the precise $\fP$-adic valuation for the algebraic part of the $L$-value $L(\ov{\phi},1)$. This precise $\fP$-adic valuation will play a key role as the base case of an induction argument in Section~\ref{S6}.

We recall that $F_\infty = K(B_{\fP^\infty})$.  A large part of this section, namely the construction of the $\fP$-adic $L$-function for~$B$ over~$F_\infty/K$ and the proof of the main conjecture for~$B$ for the tower $F_\infty/F$, has been established in \cite{CL}. We will just briefly recall these results and use the results from Sections \ref{S3} and \ref{S4} in order to obtain the desired $\fP$-adic valuation for the algebraic $L$-value of~$\phi$.

We define $M(F_\infty)$ to be the maximal abelian $2$-extension of~$F_\infty $ which is unramified outside the primes lying above~$\fp$, and we put
$$
X(F_\infty)  = \Gal(M(F_\infty)/F_\infty).
$$
By maximality,  $M(F_\infty)$ is Galois over~$K$, and $\msg = \Gal(F_\infty/K)$ acts on $X(F_\infty)$ in the usual manner via lifting of inner automorphisms. We list some results relating $X(F_\infty)$ to the Selmer groups in the following theorem, whose proof can be found in
\cite[Theorem 3.9, Proposition 3.10 and Proposition 3.11]{CL}.

\begin{thm}\label{sel-rel-x}

\begin{enumerate}
  \item We have
  $$\Sel_{\fP^\infty}(B/F_\infty)= \Sel'_{\fP^\infty}(B/F_\infty) = \Hom(X(F_\infty), B_{\fP^\infty}),$$
  where all these modules are Galois modules over~$\msg$.
  \item The Selmer group $\Sel'_{\fP^\infty}(B/F_\infty)$ satisfies
  $$\Sel'_{\fP^\infty}(B/F_\infty) = \Sel'_{\fP^\infty}(B/F_\infty)^\Delta,$$
  where $\Delta=\Gal(F_\infty/K_\infty)$.
  \item The restriction map yields the isomorphism
\[\Sel'_{\fP^\infty}(B/F)\simeq \Sel'_{\fP^\infty}(B/F_\infty)^{\Gamma},\]
where $\Gamma=\Gal(F_\infty/F)$.
\end{enumerate}
\end{thm}

Let~$\Lambda(\Gamma)$ be the Iwasawa algebra of~$\Gamma$
with coefficients in~$\BZ_p$. As usual, we identify~$\Lambda$ with the power series ring $\BZ_p[[T]]$ by fixing a topological generator $\gamma$ of~$\Gamma$.
For any finitely generated torsion $\Lambda(\Gamma)$-module, we denote by~$c_\fM(T)$ the characteristic power series of~$\fM$ (defined up to a unit in~$\BZ_p[[T]]$). For a $\Lambda$-module~$\fM$, we Let~$\fM_\Gamma$ be the maximal quotient $\Lambda$-module of~$\fM$ such that $\Gamma$ acts trivially. Denote by~$\fM^\Gamma$ the $\Lambda$-submodule of~$\fM$ consisting of elements fixed by~$\Gamma$.
We recall the following elementary Lemma~\cite[Appendix A.2]{CS}.

\begin{lem}\label{t3.5} Let~$\fM$ be a finitely generated torsion $\Lambda(\Gamma)$-module, and let~$c_\fM(T)$ be any characteristic power series of~$\fM$. Then the following assertions are equivalent: (i) $c_\fM(0) \neq 0$, (ii) $\fM_\Gamma$ is finite, and (iii) $\fM^\Gamma$ is finite. Moreover, when these equivalent assertions hold, we have
\[
|c_\fM(0)|_p^{-1} = \#(\fM_\Gamma)/\#(\fM^\Gamma).
\]
Here, $|\cdot|_p$ is a fixed multiplicative valuation on $\BC_p$ extending $|\cdot|_\fP$. We note here that we always fix an embedding of~$T$ into $\BC_p$ via the prime~$\fP$.
\end{lem}

\noindent In particular,  if $\fM$ has no non-zero finite $\Gamma$-submodules, then the lemma shows that $\fM^\Gamma = 0$ when $c_\fM(0) \neq 0$. Finally,  we write $\msi$ for the ring of integers of the completion of the maximal unramified extension of~$T_\fP$, and we define \mbox{$\Lambda_\msi(\Gamma) = \msi[[\Gamma]]$}
to be the Iwasawa algebra of~$\Gamma$ with coefficients in~$\msi$.

\medskip

Another key input we need for treating the case $p=2$ is that we need to show $X(F_\infty)$ is a finitely generated $\BZ_2$-module. It follows easily from our earlier work~\cite{CKL}.

\begin{prop} The Iwasawa module $X(F_\infty)$ is a finitely generated $\BZ_2$-module. In particular, $X(F_\infty)$ is a finitely generated torsion $\Lambda$-module.
\end{prop}

For simplicity, we write $c_X(T)$ for a characteristic power series of~$X(F_\infty)$. Let
\[
\rho_{\fP}: \msg \to \BZ_2^\times
\]
be the character giving the action of the Galois group $\msg$ on $B_{\fP^\infty}$. Let~$\mathrm{T}_\fP B$ be the $\fP$-adic Tate module of~$B$.
Put
$$\mathrm{T}_\fP B(-1) = \Hom(B_{\fP^\infty}, T_\fP/\CO_{T_\fP}).$$
If $\fM$ is any finitely generated torsion $\Lambda(\Gamma)$-module, we define
$$\fM(-1) = \fM\otimes_{\CO_\fP}(\mathrm{T}_\fP B(-1)).$$
Then $\Hom(X(F_\infty), B_{\fP^\infty})$ is isomorphic to
$$\Hom(X(F_\infty)(-1), T_\fP/\CO_{T_\fP})$$
as $\Gamma$-modules. Since $T_\fP/\CO_{T_\fP}=\BQ_2/\BZ_2$ and Theorem~\ref{sel-rel-x}, the Pontryagin dual of~$\Sel_{\fP^\infty}(B/F_\infty)$ is equal to $X(F_\infty)(-1)$. Denote by~$\rho_{\fP,\Gamma}$ the restriction of~$\rho_\fP$ to~$\Gamma$.

\begin{prop}\label{t3.6} Let~$c_X(T)$ be a characteristic power series of~$X(F_\infty)$, and let~$u = \rho_{\fP, \Gamma}(\gamma)$. Then $c_X(u-1) \neq 0$. Moreover, we have
\begin{equation}\label{t3.7}
|c_X(u - 1)|_p^{-1} = \frac{1}{2}  \left|\left(1-\frac{\phi(\fp)}{2}\right)\right|_\fP^{-1}.
\end{equation}
\end{prop}

\begin{proof}
By Theorem~\ref{sel-rel-x}, $X(F_\infty)(-1)_\Gamma$ is the Pontryagin dual of~$\Sel'_{\fP^\infty}(B/F)^\Delta$. By Propositions \ref{t2.2} and \ref{h-seq}, we know that $X(F_\infty)(-1)_\Gamma$ is finite.
It is clear that the characteristic power series of the $\Lambda(\Gamma)$-module $X(F_\infty)(-1)$ is given by \mbox{$c_X(u(1+T) - 1)$}. Applying Lemma~\ref{t3.5}, the first assertion follows.
The remaining part of the proposition follows from a theorem of Greenberg, which assers that $X(F_\infty)(-1)$ has no non-zero finite $\Lambda$-submodules.
\end{proof}

Recall that we have chosen a global minimal generalised Weierstrass equation \eqref{1} for~$A/H$. Moreover, since $H = K(j(\CO_K))$ and we have fixed an embedding of~$K$ into $\BC$, we also have an embedding of~$H$ into $\BC$. The N\'{e}ron differential $\omega=dx/(2y+a_1x+a_3)$ on $A/H$ then has a complex period lattice of the form $\CL=\Omega_\infty \CO_K$, where $\Omega_\infty$ is a nonzero complex number which is uniquely determined up to sign.

We now fix an embedding of~$H$ into the fraction field of~$\msi$, which amounts to choosing a prime~$w$ of~$H$ lying above~$\fp$. We write $\wh{A}$ for the formal group of~$A$ under this embedding. Since $\wh{A}$ has height $1$ as a formal group, there exists an isomorphism
\begin{equation}\label{4.2}
  j_\fp: \wh{\BG}_m\xto{\sim} \wh{A}
\end{equation}
 of formal groups over~$\msi$, where $\wh{\BG}_m$ denotes the formal multiplicative group. As usual, we take as a parameter for~$\wh{A}$ the local parameter $t=-\frac{x}{y}$ at infinity of the Weierstrass equation \eqref{1}. The isomorphism $j_\fp$ is then given by a power series $t=j_\fp(w_0)$ with coefficients in~$\msi$, where $w_0$ is the parameter of~$\wh{\BG}_m$. We then define the $\fp$-adic period $\Omega_\fp$ to be the coefficient of~$w_0$ in the formal power series $t=j_\fp(w_0)$. Since $j_\fp$
 is an isomorphism, $\Omega_\fp$ is a unit in~$\msi$.

We fix an embedding of the field $T$ into $\BC$ which extends our fixed embedding of~$K$ into $\BC$. We define the Hecke $L$-series
\[
L(\ov{\phi}^k, s)=\sum_{ (\fb,\fq)=1}\frac{\ov{\phi}(\fb)}{N(\fb)^s},
\]
where the sum is taken over all integral ideals $\fb$ of~$K$ which are prime to $\fq$. It is classical
that the Dirichlet series on the right converges for~$\mathrm{Re}(s) > 3/2$, and it has analytic continuation to the whole complex plane.
Finally, we fix an embedding of the compositum $HT$ into the fraction field of~$\msi$ which induces the prime~$w$ of~$H$ and the prime~$\fP$ of~$T$.
This is possible because $H\cap T = K$.

The following proposition gives the $\fP$-adic $L$-functions and the Iwasawa main conjecture for the tower $F_\infty/F$. A detailed proof can be found in \cite[Theorems 5.4 and 7.13, Sections 4--7]{CL}.

\begin{prop}\label{t4.1}
For all odd integers $k\geq1$, we have
\[\frac{L(\ov{\phi}^k,k)}{\Omega^{k}_\infty}\in HT.\]
There exists a unique measure $\mu_A$ in~$\Lambda_\msi(\Gamma)$
such that, for all odd positive integers $k = 1, 3, 5, \ldots$, we have
\begin{equation}\label{4.4}
\Omega_\fp^{-k}\int_{\Gamma} (\rho_{\fP, \Gamma})^kd\mu_A = (k-1)!\Omega_\infty^{-k}L(\ov{\phi}^k, k)(1-\phi^k(\fp)/2).
\end{equation}
Furthermore, $\mu_A$ satisfies the main conjecture
\[
c_X\Lambda_\msi(\Gamma) = \mu_A \Lambda_\msi(\Gamma),
\]
where  $c_X$ is a characteristic element for~$X(F_\infty)$.
\end{prop}

Combined with Proposition~\ref{t3.6}, we thus obtain a new proof of the theorem of Rohrlich on the non-vanishing of~$L(A,1)$ on noting that $L(A,1)\neq0$ if and only if $L(\ov{\phi},1)\neq0$. Furthermore, we obtain the precise $2$-adic valuation of the algebraic part of the $L$-value $L(\ov{\phi},1)/\Omega_\infty$:
\bigskip

\begin{cor}\label{t4u.2} We have $L(\ov{\phi},1) \neq 0$. Moreover, we have
\begin{equation}\label{4.6}
\left|\frac{L(\ov{\phi},1)}{\Omega_\infty}\right|_2^{-1}=2^{-1}.
\end{equation}
\end{cor}

\begin{proof}
By Theorem~\ref{t4.1}, we have
$$c_X = \mu_A\beta$$
for some unit $\beta$ in~$\Lambda_\msi(\Gamma)$.
Note that the character $\rho_{\fP,\Gamma}$ can be linearly extended to an algebra homomorphism from $\Lambda_\msi(\Gamma)$ to $\msi$.
Applying $\rho_{\fP,\Gamma}$ on both sides of the above equation and
recalling that $u = \rho_{\fP, \Gamma}(\gamma)$, we obtain
\[
|c_X(u-1)|_2^{-1} = |\mu_A(u-1)|_2^{-1}  = \left|\int_\Gamma\rho_{\fp, \Gamma} d\mu_A\right|_2^{-1}.
\]
In particular, we conclude from \eqref{4.4} for~$k=1$ that $c_X(u-1) \neq 0$ if and only if $L(\ov{\phi},1) \neq 0$. Thus, the first assertion of the corollary follows from the first assertion of Proposition~\ref{t3.6}. Finally, since $\Omega_\fp$ is a unit in
$\msi$, the formula \eqref{4.6} follows immediately from \eqref{t3.7} and \eqref{4.4} for~$k=1$.
\end{proof}

We end this Section~with some remarks on the algebraic $L$-value in Corollary~\ref{t4u.2}. It was shown by Buhler and Gross (or in \cite{CL}) that the fractional ideal $\frac{L(\ov{\phi},1)}{\Omega_\infty}\CO_{HT}$ descends to a fractional ideal $\fm$ in~$T$. Corollary~\ref{t4u.2} shows the $\fP$-adic valuation of~$\fm$ is equal to $-1$. Thus, we have proved:

\begin{thm}
The $\mathfrak{P}$-part of the refined Birch--Swinnerton-Dyer formula holds for~$B$.
\end{thm}

\section{Generalised Zhao's method and $2$-adic valuation of central $L$-values}\label{S6}

In this section, we will use an extension of Zhao's induction method \cite{Zhao1, Zhao2}, modified for the abelian varieties, to obtain the exact valuation of central $L$-values for quadratic twists, using Corollary~\ref{t4u.2} as the base case. The detailed proof goes in a similar way as in \cite[Section~\S9]{CL}, for the reader's convenience, we will sketch the proof in the following part of the section.

As defined in the introduction, Let~$R\in\fR$ be of the form $R = r_1\cdots r_k$, where $k\geq 0$ and $r_1, \dots, r_k$ are distinct primes such that (i) $r_i \equiv 1 \mod 4$, and  (ii) $r_i$ is inert in~$K$ for~$i=1,\ldots, k$. We recall that $\phi$ is the Hecke character of~$B/K$. Then, since $(R, q) = 1$, for each positive integer $d\mid R$, the Hecke character $\phi_d$ of~$B^{(d)}/K$ is given by~$\phi_d=\phi\chi_d$, where $\chi_d$ denotes the abelian character of~$K$ defining the quadratic extension $K(\sqrt{d})/K$. Here, as usual, we denote by~$B^{(d)}$ the twist of~$B$ by the extension $K(\sqrt{d})/K$. The remainder of this Section~will be dedicated to concluding the proof of the following main result stated in the introduction.

\begin{thm}\label{main2}
For any $R = r_1\cdots r_k \in \fR$. Then $\frac{L(\overline{\phi}_R,1)\sqrt{R}}{\Omega_\infty}\in TH$, and for any prime~$\mathcal{P}$ of~$TH$ lying above~$\fp$, we have
\begin{align*}\ord_{\mathcal{P}}\left(\frac{L(\overline{\phi}_R,1)\sqrt{R}}{\Omega_\infty}\right)=k-1.
\end{align*}
In particular,  $L(A^{(R)}/H, 1) \neq 0$.
\end{thm}

We define the imprimitive Hecke $L$-series
\[L_R(\ov{\phi}_d, s)=\sum_{ (\fb,R\fq)=1}\frac{\ov{\phi}_d(\fb)}{N(\fb)^s},\]
where the sum on the right is taken over all integral ideals $\fb$ of~$K$ which are prime to $R\fq$. 

We define the fields
\begin{equation}\label{6}
J_R = K(\sqrt{r_1}, \dots, \sqrt{r_k}), \, \, H_R = H(\sqrt{r_1}, \dots, \sqrt{r_k}),
\end{equation}
and we recall that $\Omega_\infty$ is the complex period defined in Section~\ref{S5}. The proof of the following proposition can be found in Proposition 9.8 and Theorem 9.9 of \cite{CL}.

\begin{prop}\label{key-identity}
 There exists an element $V_R\in H_R$ which is integral at all places of~$H_R$ above~$2$, and which satisfies
\[
  \sum_{d\mid R}L_R(\ov{\phi}_d,1)/\Omega_\infty=2^k  V_R,
\]
where the sum runs over all positive divisors $d$ of~$R$.
\end{prop}

We list the following two lemmas, whose proofs are identical to those of \cite[Lemma 9.13]{CL} and \cite[(4.6) in Theorem 4.2]{choi}, respectively.

\begin{lem}
For each $R \in \fR$, the extension $J_R/K$ defined by \eqref{6} is unramified at the primes of~$K$ lying above 2. \end{lem}

\begin{lem}\label{4.3} Let~$d$ be any positive divisor of~$R$, and Let~$r$ be any prime dividing~$R$ with $(r, d) = 1$. Then $\phi_d(r\CO_K) = - r$. \end{lem}

For each positive divisor $d$ of~$R$, we define
$$
\msl(d) = \sqrt{d}L(\ov{\phi}_d, 1)/\Omega_\infty, \, \, \msl = \msl(1).
$$
Recall from Corollary~\ref{t4u.2}, we always have $\msl \neq 0$ for all primes $q \equiv 7 \mod 8$. From \cite[Proposition 9.5]{CL},
we have

\begin{prop}\label{4.7} Assume $R \in \fR$, and Let~$d$ be any positive divisor of~$R$. Then $\msl(d)/\msl$ belongs to the field $T$.
\end{prop}

We can now conclude the proof of Theorem~\ref{main2} using induction on the number~$k$ of prime factors of~$R$. Assume first that $k=1$, so that $R = r$, a prime number. Then Proposition~\ref{key-identity} gives
\begin{equation}\label{4.8}
\msl(r)/\sqrt{r}+ (1-\ov{\phi}((r))/r^2)\msl = 2V_r.
\end{equation}
Let~$\mathfrak{Q}$ be any prime of~$HT(\sqrt{r})$ lying above the prime~$\fP$ of~$T$. We then have \mbox{$\ord_\mathfrak{Q}(V_r) \geq 0$}. Furthermore, by Lemma~\ref{4.3}, we have \mbox{$(1-\ov{\phi}((r))/r^2) = (1 + 1/r)$}. Since $r+1 \equiv 2 \mod 4$, it follows from Corollary~\ref{t4u.2} that \mbox{$\ord_\mathfrak{Q}((1-\ov{\phi}((r))/r^2)\msl) = 0$}. As $\ord_\mathfrak{Q}(V_r) \geq 0$, we conclude from \eqref{4.8} that $\ord_\mathfrak{Q}(\msl(r)/\sqrt{r}) = 0$, and so Theorem~\ref{main2} holds when   $k=1$.

Suppose now that $R = r_1\cdots r_k$ with $k \geq 2$. By Proposition~\ref{key-identity}, we have
\begin{equation}\label{4.9}
\msl(R)/\sqrt{R} + \sum_{d|R, d \neq 1, R}\Lambda(d, R)/\sqrt{d} + \msl \prod_{i=1}^{k}(1- \ov{\phi}((r_i))/r_i^2) = 2^kV_R,
\end{equation}
where
$$
\Lambda(d, R) = \msl(d) \prod_{r|R/d}(1-\ov{\phi}_d((r))/r^2).
$$
The problem here is that the terms $\Lambda(d, R)$ lie in an extension of~$HT$, where the prime~$\fP$ of~$T$ is unramified but will usually have a large residue class field extension. This means that one cannot carry through the inductive argument in its naive form, and we must appeal to Proposition~\ref{4.7} to get around it. The key to overcoming this difficulty is to divide both side of \eqref{4.9} by the non-zero number $\msl$. Then defining $\Phi(d, R) = \Lambda(d, R)/\msl$ for each positive integer divisor $d$ of~$R$, we obtain the equation
\begin{equation}\label{4.10}
\Phi(R)/\sqrt{R} + \sum_{d|R, d \neq 1, R}\Phi(d, R)/\sqrt{d}  + \prod_{i=1}^{k}(1- \ov{\phi}((r_i))/r_i^2) = 2^kV_R/\msl,
\end{equation}
where $\Phi(R) = \msl(R)/\msl$. Let~$H_R$ be the field defined in \eqref{6}, and we now take $\mathfrak{Q}$ to be any prime of the compositum $H_RT$ lying above~$\fp$ so that $\mathfrak{Q}/\fP$ is unramified. By Proposition~\ref{key-identity}, we have  $\ord_\mathfrak{Q}(V_R) \geq 0.$ Thus we conclude from Corollary~\ref{t4u.2} that $\ord_\mathfrak{Q}(2^kV_R/\msl) \geq k+1$. Thanks to  Lemma~\ref{4.3}, we have
\begin{equation} \label{4.11}
\ord_\mathfrak{Q}(\prod_{i=1}^{k}(1- \ov{\phi}((r_i))/r_i^2)) = k.
\end{equation}
On the other hand, our inductive hypothesis together with Corollary~\ref{t4u.2} and Lemma~\ref{4.3} shows that, for each positive divisor $d$ of~$R$ with $d \neq 1, R$, we have
\[
\ord_\mathfrak{Q}(\Phi(d, R)/\sqrt{d}) = k.
\]
Of course, these estimates alone do not allow us to conclude that $\ord_\mathfrak{Q}(\Phi(R)/\sqrt{R}) = k$ when applied to \eqref{4.10}. However, the argument is saved by Proposition~\ref{4.7},
which tells us that $\Phi(d, R)$ belongs to the field $T$ for every positive divisor $d$ of~$R$, and so it lies in the completion $T_\fP$ at $\fP$. Since $\fP$ has its residue degree of order 2, we can write for each divisor $d$ of~$R$ with $1<d<R$,
\[
\Phi(d, R)/\sqrt{d} = \sqrt{d}\pi_\fP^k(1 + \pi_\fP b_d),
\]
where $\pi_\fP$ is a local parameter at $\fP$ and $\ord_\fP(b_d) \geq 0$. Thus,
\[
\sum_{d|R, d \neq 1, R}\Phi(d, R)/\sqrt{d}  \equiv \pi_\fP^kD_R \mod \mathfrak{Q}^{k+1}
\]
with  $D_R = \sum_{d|R. d \neq 1, R}\sqrt{d}$. But
$$
D_R^2 \equiv  \sum_{d|R. d \neq 1, R} d \mod \mathfrak{Q}
$$
and $\sum_{d|R. d \neq 1, R} d \equiv 2^k \mod 2$, whence $\ord_\mathfrak{Q}(D_R) \geq 1$. Thus, we have shown that
$$
\ord_\mathfrak{Q}\left(\sum_{d|R, d \neq 1, R}\Phi(d, R)/\sqrt{d} \right) \geq k+1.
$$
It now follows from \eqref{4.10} and \eqref{4.11} that $\ord_\mathfrak{Q}(\Phi(R)) = k$.  Thus, again applying Corollary~\ref{t4u.2}, we
have finally completed the proof of Theorem~\ref{main2} by induction on the number of prime factors of~$R \in \fR$.

We will conclude this paper the proof of the density result in Corollary~\ref{cor1.2}.
The proof follows closely the ideas of Serre \cite{Serre}, which is based on generalisations of the Tauberian theorem of Ikehara due to Delange\cite{Delange}, Wintner\cite{Wintner} et al.\,and a method of Landau \cite{Landau}. Such an argument has already appeared in, for example, the works of Ono \cite{Ono} or Kriz--Li \cite{Kriz-Li}, but we write out the details for our particular case.

\begin{cor}[Corollary~\ref{cor1.2}] Let $D\geq 1$ denote a fundamental discriminant, and Let~$N(X) = \#\{\text{$D < X$ : $L(A^{(D)}/H,1)\neq 0\}$}$. Then we have
$$N(X)>\!\!>\frac{X}{\log^{\frac{3}{4}} X}.$$
\end{cor}

\begin{proof}

We define the set of primes
$\mathcal{P}$ denote the set of primes which is the \textit{complement} of the set of primes which are congruent to $1$ modulo $4$ and inert in~$K$. Then~$\mathcal{P}$ is \textit{regular} of density $1-\frac{1}{2}\cdot\frac{1}{2}=\frac{3}{4}$ in the sense of Delange \cite{Delange}, by the Chebotarev density theorem. Furthermore, $\mathcal{P}$ is \textit{associated} to the set
$E=\mathbb{N}_{>0}\backslash \mathfrak{R}$, in the sense that for any $p\in \mathcal{P}$ and any integer $m\geq 1$ not divisible by~$p$, we have $pm\in E$. Given $X$, Let~$E(X)=\{D\leq X: D\in E\}$ and $\mathfrak{R}(X)=\{D\leq X: D\in \mathfrak{R}\}$.

Noting that now $\mathfrak{R}(X)=E'(X)$ in the notation of \cite{Serre}, we have from \cite[Th\'{e}or\`{e}me 2.8]{Serre}, for any integer $k\geq 0$, real numbers $c_0, c_1,\ldots  , c_k$ with $c_0>0$  such that
\begin{align*}\#(\mathfrak{R}(X))&=\frac{X}{\log^{\frac{3}{4}} X} (c_0+c_1/\log X +\cdots + c_k \log^k X+O(1/\log^{k+1}X)).
\end{align*}
In particular,
\begin{align*}\#(\mathfrak{R}(X))&=c_0\frac{X}{\log^{\frac{3}{4}} X} +O(X/\log^{\frac{7}{4}}X).
\end{align*}
The result now follows, on noting that by Theorem~\ref{main} we have $L(A^{(D)}/H,1)\neq 0$ for all $D\in \mathfrak{R}$.
\end{proof}

\bigskip

\address{\textit{Yukako Kezuka}, Kanazawa University, Natural Science and Technology, Kakumamachi, Kanazawa, 920-1164, Ishikawa, Japan.}
\email{\it kezuka@se.kanazawa-u.ac.jp}

\medskip

\address{\textit{Yong-Xiong Li}, Beijing Institute of Mathematical Sciences and Applications, No.\,544, Hefangkou Village Huaibei Town, Huairou District Beijing 101408, Beijing, China.}
\email{\it yongxiongli@gmail.com}


\begin{thebibliography}{9}

\bibitem{A} \textit{N.\,Arthaud}, On Birch and Swinnerton-Dyer's conjecture for elliptic curves with complex multiplication, I, Compos.\,Math.\,37 (1978), pp.\,209--232.

\bibitem{bas} \textit{M.\,Basmakov}, Cohomology of Abelian varieties over a number field,  Uspehi Mat.\,Nauk \textbf{27} (1972), no. 6 (168), pp.\,25--66.

\bibitem{BG}\textit{J.\,Buhler} and \textit{B.\,Gross},  Arithmetic on elliptic curves with complex multiplication. II,  Invent.\,Math.\,79 (1985), pp.\,11--29.

\bibitem{BF} \textit{A.\, Burungale} and \textit{M.\,Flach},   The conjecture of Birch and Swinnerton-Dyer for certain elliptic curves
with complex multiplication, Camb.\, J.\, Math. 12, No. 2, 357--415 (2024).

\bibitem{BT} \textit{A.\, Burungale} and \textit{Y.\,Tian},  A rank zero $p$--converse to a theorem of Gross-Zagier, Kolyvagin and
Rubin, preprint.

\bibitem{BT2} \textit{A.\, Burungale} and \textit{Y.\,Tian},  The even parity Goldfeld conjecture: congruent number elliptic curves. J. Number Theory 230 (2022), 161--195. 

\bibitem{choi}\textit{J.\,Choi},  On the $2$-adic valuations of central $L$-values of elliptic curves, J. Number Theory 204, 405--422 (2019).


\bibitem{CKL}\textit{J.\,Choi}, \textit{Y.\,Kezuka} and \textit{Y.\,Li}, Analogues of Iwasawa's $\mu=0$ conjecture and weak Leopoldt theorem for certain non-cyclotomic $\mathbb{Z}_2$-extensions,  Asian J.\,Math.\,\textbf{23} (2019), no.\,3, pp.\,383-400.

\bibitem{Coa}\textit{J.\,Coates},  Infinite Descent on Elliptic Curves with Complex Multiplication, Arithmetic and geometry,
Progress in Mathematics \textbf{35}, M.\,Artin and J.\,Tate (eds) Birkh\"{a}user Boston, Boston, MA (1983) pp.\,107--137.


\bibitem{CL} \textit{J.\,Coates} and \textit{Y.\,Li},  Non-vanishing theorems for central $L$-values of some elliptic curves with complex multiplication,  Proc.\,Lond.\,Math.\,Soc. (3) \textbf{121} (2020), no.\,6, pp.\,1531--1578.

\bibitem{CS}\textit{J.\,Coates} and \textit{R.\,Sujatha},  Cyclotomic Fields and Zeta Values, Springer, first edition (2006).

\bibitem{CW}\textit{J.\,Coates} and \textit{A.\,Wiles},  On the conjecture of Birch and Swinnerton-Dyer, Invent.\,Math., 39 (1977), pp.\,223--251.

\bibitem{Delange}\textit{H.\,Delange}, Sur la distribution des entiers ayant certaines propriétés, Ann. Sci. École Norm. Sup. (3) 73 (1956), 15--74.

\bibitem{DJS} \textit{A.\, Dabrowski}, \textit{T.\, Jedrzejak}, and  \textit{L.\, Szymaszkiewicz}, {\em Critical $L$-values for some quadratic twists of Gross curves.},  Asian J. Math. 24 (2020), no. 2, 257--265.

\bibitem{Greenberg0}  \textit{R.\,Greenberg},  On the structure of certain Galois groups,  Invent.\,Math.\,\textbf{47} (1978), no.\,1, pp.\,85--99.


\bibitem{Gross78} \textit{B.\,Gross},  Arithmetic on elliptic curves with complex multiplication, Lecture Notes in Mathematics 776, Springer, Berlin, 1980.

\bibitem{Gross82} \textit{B.\,Gross},  Minimal models for elliptic curves with complex multiplication, Compositio Math. \textbf{45} (1982), pp.\,155--164.

\bibitem{gross-bsd}\textit{B.\,Gross},   On the conjecture of Birch and Swinnerton-Dyer for elliptic curves with complex multiplication,   Number theory related to Fermat's last theorem (Cambridge, Mass., 1981), pp.\,219--236, Progr.\,Math., 26, Birkhauser, Boston, Mass., 1982.

\bibitem{gon-bsd}\textit{C.\,Gonzalez-Aviles},  On the conjecture of Birch and Swinnerton-Dyer,  Trans.\, Am.\, Math.\, Soc. 349, No. 10, 4181-4200 (1997).



\bibitem{Kriz-Li} \textit{D.\,Kriz} and \textit{C.\,Li}, Goldfeld's conjecture and congruences between Heegner points, Forum Math.\,Sigma \textbf{7} (2019), e15, 80.

\bibitem{Landau} \textit{E.\,Landau} \"{U}ber die Einteilung der positiven ganzen Zahlen in vier Klassen nach der Mindestzahl der zu ihrer additiven Zusammensetzung erforderlichen Quadrate, Arch.\,der Math.\,und Phys., (3) \textbf{13} (1908), pp.\,305--312.

\bibitem{JLi} \textit{J.\,Li}, On the $2$-adic logarithm of units of pure imaginary quartic fields, Asian J.\,Math.\, \textbf{25} (2021) pp.\,177--182.

\bibitem{yli} \textit{Y.\,Li}, Some remarks on  $p$-adic  $L$-functions for certain abelian varieties with complex multiplication, Int.\, J.\, Number Theory 20, No. 1, 159--184 (2024).


\bibitem{Ono} \textit{K.\,Ono}, Non-vanishing of quadratic twists of modular $L$-functions and applications to elliptic curves, J.\,Reine Angew.\,Math.\,\textbf{533} (2001), pp.\,81–97.


\bibitem{razar} \textit{M.\,Razar}, The non-vanishing of~$L(1)$ for certain elliptic curves with no first descents.,  Am. J. Math. 96, 104-126 (1974).

\bibitem{rv} \textit{F.\,Rodriguez Villegas}, On the square root of special values of certain $L$-series. Invent. Math. 106 (1991), no. 3, 549--573. 

\bibitem{Ro1} \textit{D.\,Rohrlich}, The non-vanishing of certain Hecke $L$-functions at the center of the critical strip,   Duke Math.\,J. \textbf{47} (1980), pp.\,223--232.

\bibitem{Ru1} \textit{K.\,Rubin} Elliptic curves with complex multiplication and the conjecture of Birch and Swinnerton-Dyer, Invent.\,Math.\,64 (1981), pp.\,455--470.

\bibitem{Serre}\textit{J-P.\,Serre}, Divisibilit\'{e} de certaines fonctions arithm\'{e}tiques, S\'{e}minaire Delange--Pisot--Poitou. Th\'{e}orie des nombres 16.1 (1974-1975), pp.\,1--28.

\bibitem{ye}\textit{Y.\,Tian}, Congruent numbers and Heegner points, Camb. J. Math. 2, No. 1, 117--161 (2014).


\bibitem{ye2} \textit{Y.\,Tian}, The congruent number problem and elliptic curves. ICM—International Congress of Mathematicians. Vol. 3. Sections 1-4, 1990--2010, EMS Press, Berlin, 



\bibitem{wal}\textit{J.-L.\,Waldspurger}, Sur les valeurs de certaines fonctions  L
  automorphes en leur centre de symétrie, Compos.\, Math.\, 54, 173--242 (1985).


\bibitem{Wintner}\textit{A.\,Wintner}, On the prime number theorem, Amer.\,J.\,Math.\,\textbf{64}, 1942, pp.\,320--326

\bibitem{yang1}\textit{T.\,Yang}, Theta liftings and Hecke $L$-functions, J.\,reine angew.\,Math. \textbf{485} (1997), pp.\,25--53.

\bibitem{yang2}\textit{T.\,Yang},  Nonvanishing of central Hecke $L$-values and rank of certain elliptic curves,  Compositio Math. \textbf{117} (1999), no.\,3, pp.\,337--359.


\bibitem{Zhao1} \textit{C.\,Zhao},  A criterion for elliptic curves with lowest 2-power in L(1). Math. Proc. Cambridge Philos. Soc. 121 (1997), no. 3, 385--400. 

\bibitem{Zhao2} \textit{C.\,Zhao},  A criterion for elliptic curves with second lowest 2-power in L(1). Math. Proc. Cambridge Philos. Soc. 131 (2001), no. 3, 385--404.






\end{thebibliography}
\end{document}